\documentclass[12pt]{amsart}

\usepackage{amssymb,latexsym,amscd}
\usepackage{amsthm,amsmath,mathrsfs}

\theoremstyle{remark}

\theoremstyle{definition}

\newcommand{\cO}{\mathcal{O}}
\newcommand{\cE}{\mathcal{E}}

\newcommand{\cF}{\mathcal{F}}
\newcommand{\cM}{\mathcal{M}}

\newcommand{\cB}{\mathcal{B}}

\newcommand{\cL}{\mathcal{L}}
\newcommand{\cJ}{\mathcal{J}}

\newcommand{\cI}{\mathcal{I}}
\newcommand{\cT}{\mathcal{T}}
\newcommand{\cQ}{\mathcal{Q}}

\newcommand{\Prym}{\mathrm{Prym}}
\newcommand{\Sym}{\mathrm{Sym}}
\newcommand{\Pic}{\mathrm{Pic}}

\newcommand{\Aut}{\mathrm{Aut}}

\newcommand{\Spec}{\mathrm{Spec}}

\newcommand{\lra}{\longrightarrow}

\newcommand{\ra}{\rightarrow}

\newcommand{\AAA}{\mathcal{A}}

\newcommand{\ZZ}{\mathbb{Z}}

\newcommand{\G}{{\rm G}}
\newcommand{\QQ}{\mathbb{Q}}

\newcommand{\CC}{\mathbb{C}}

\newcommand{\End}{\mathrm{End}}

\newcommand{\Hom}{\mathrm{Hom}}

\newcommand{\GL}{\mathrm{GL}}

\newcommand{\SL}{\mathrm{SL}}

\newcommand{\Nm}{\mathrm{Nm}}

\newcommand{\im}{\mathrm{im}}
\newcommand{\rk}{\mathrm{rk}}

\newcommand{\coker}{\mathrm{coker}}

\def\map#1{\ \smash{\mathop{\longrightarrow}\limits^{#1}}\ }
\newcommand{\calM}{{\mathcal M}}
\newcommand{\calN}{{\mathcal N}}

\theoremstyle{plain}

\newtheorem{thm}{Theorem}[section]

\newtheorem{lem}[thm]{Lemma}

\newtheorem{prop}[thm]{Proposition}

\newtheorem{cor}[thm]{Corollary}

\newtheorem{rem}[thm]{Remark}

\newtheorem{defi}[thm]{Definition}

\begin{document}

\title[]{Prym varieties of spectral covers}

\begin{abstract}
Given a possibly reducible and non-reduced spectral cover $\pi: X \ra C$ over a smooth projective complex curve $C$ we determine the group of
connected components of the Prym variety $\Prym(X/C)$. As an immediate application we show that the finite group  of $n$-torsion points of the Jacobian of $C$ acts trivially on the cohomology of the twisted $\SL_n$-Higgs moduli space up to the degree which is predicted by topological mirror symmetry. In particular this yields a new proof  of a result of Harder--Narasimhan, showing 
that this finite group acts trivially on the cohomology of the twisted $\SL_n$ stable bundle moduli space.  \end{abstract}

\author{Tam\'as Hausel}

\author{Christian Pauly}

\address{Mathematical Institute \\ University of Oxford  \\ 24-29 St Giles' \\ Oxford OX1 3LB \\ United Kingdom }

\email{hausel@maths.ox.ac.uk}

\address{D\'epartement de Math\'ematiques \\ Universit\'e de Montpellier II - Case Courrier 051 \\ Place Eug\`ene Bataillon \\ 34095 Montpellier Cedex 5 \\ France}

\email{pauly@math.univ-montp2.fr}




\maketitle

\section{Introduction}

Recently there has been renewed interest in the topology of the Hitchin fibration. The Hitchin fibration is  an integrable system associated to a complex reductive group $\rm G$ and a smooth complex projective curve $C$. It was introduced by Hitchin \cite{hitchin} in 1987, originating in his study of 
a $2$-dimensional reduction of the Yang-Mills equations. In 2006, Kapustin and Witten \cite{kapustin-witten} highlighted the importance of the Hitchin fibration  for $S$-duality and the Geometric Langlands program. While the work of Ng\^{o}
\cite{ngo2} in 2008 showed that the topology of the Hitchin fibration is responsible for the fundamental lemma in the Langlands program. In Ng\^{o}'s work and later in the work of Frenkel and Witten \cite{FW} a certain symmetry of the Hitchin fibration plays an important role. 

\bigskip

In this paper we focus on the Hitchin fibration for the group ${\rm G} = \SL_n$ and for a line 
bundle $M$ over $C$, i.e.,
the morphism
\begin{equation} \label{Hitchinfibration}
 h:\calM \lra \AAA_n^0 = \bigoplus_{j=2}^n H^0(C, M^j). 
\end{equation}
Here $\calM$ denotes the quasi-projective moduli space of semi-stable Higgs bundles $(E, \phi)$ over $C$ of rank $n$, fixed determinant $\Delta$  and with trace-free Higgs field $\phi \in H^0(C, \End_0(E) \otimes M)$.  In the case of $\SL_n$ the above
mentioned symmetry group of the Hitchin fibration is the Prym variety of a spectral cover. For the topological applications the determination of its group of components is the first step. Ng\^{o} works with integral, that is irreducible and reduced, spectral curves; but it is interesting to extend his results to non-integral curves. For reducible but reduced spectral curves it was achieved by Chaudouard and Laumon \cite{CL}, who proved  the weighted fundamental lemma by generalizing
Ng\^{o}'s results to reduced spectral curves. In this paper we determine the group of connected 
components of the Prym variety for non-reduced spectral curves as well. 
\bigskip

In order to state the main theorem we need to introduce some notation. We associate to any spectral cover 
$\pi: X \ra C$ a finite group $K$ as follows: let $X = \bigcup_{i \in I} X_i$ be its decomposition
into irreducible components $X_i$,  let $X_i^{red}$ be the underlying reduced curve of 
$X_i$, $m_i$ the multiplicity of $X_i^{red}$ in $X_i$ and $\widetilde{X}_i^{red}$ the normalization
of $X_i$. We denote by $\widetilde{\pi}_i : \widetilde{X}_i^{red} \ra C$ the projection onto $C$ and
introduce the finite subgroups 
$$ K_i = \ker \left( \widetilde{\pi}_i^* : \Pic^0(C) \lra \Pic^0(\widetilde{X}_i^{red}) \right)
\subset \Pic^0(C), $$
as well as the subgroups $(K_i)_{m_i} = [m_i]^{-1}(K_i)$, where $[m_i]$ denotes multiplication
by $m_i$ in the Picard variety $\Pic^0(C)$ parameterizing degree $0$ line bundles over $C$. 
Finally, we put
\begin{equation} \label{defK}
K = \bigcap_{i \in I} (K_i)_{m_i} \subset \Pic^0(C). 
\end{equation}
We denote by $C_n$ the multiple curve with trivial nilpotent structure of order $n$ having underlying
reduced curve $C$.

\bigskip

We consider the norm map $\Nm_{X/C} : \Pic^0(X) \ra \Pic^0(C)$ between the connected
components of the identity elements of the Picard schemes of the curves $X$ and $C$ and 
define the Prym variety
$$ \Prym(X/C) := \ker (\Nm_{X/C}). $$
Our main result is the following

\bigskip

\begin{thm}\label{main}
Let $\pi: X \ra C$ be a spectral cover of degree $n \geq 2$. With the  notation above we have the
following results:

\begin{enumerate}
\item The group of connected components $\pi_0(\Prym(X/C))$ of the Prym variety $\Prym(X/C)$ equals
$$\pi_0(\Prym(X/C)) = \widehat{K}, $$
where $\widehat{K} = \Hom(K, \CC^*)$ is the group of characters of $K$.
\item The natural homomorphism from the group of $n$-torsion line bundles $\Pic^0(C)[n]$  to 
$\pi_0(\Prym(X/C))$ given by
$$ \Phi : \Pic^0(C)[n] \lra \pi_0(\Prym(X/C)), \qquad \gamma \mapsto [\pi^* \gamma], $$
where $[\pi^* \gamma]$ denotes the class of $\pi^* \gamma \in \Pic^0(X)$ in $\pi_0(\Prym(X/C))$
is surjective. In particular, we obtain an upper bound for the order 
$$|\pi_0(\Prym(X/C))| \leq n^{2g},$$
where $g$ is the genus of the curve $C$.
\item The map $\Phi$ is an isomorphism if and only if $X$ equals the non-reduced curve $C_n$ with
trivial nilpotent structure of order $n$.
\end{enumerate}

\end{thm}

Similar descriptions of $\pi_0(\Prym(X/C))$ were given in \cite{ngo1} in the case of integral spectral
curves and by \cite{CL} in the case of reducible but reduced spectral curves. Also \cite{dCHM}
use special cases for $\SL_2$. 

\bigskip

For a characteristic $a \in \AAA_n^0$ we denote by
$\pi : X_a \ra C$ the associated spectral cover of degree $n$ (see section \ref{sectionspectralcovers}) and by $K_a$ the
subgroup of $\Pic^0(C)$ defined in \eqref{defK} and corresponding to the cover $X_a$. 
Let $\Gamma \subset \Pic^0(C)[n]$
be a cyclic subgroup of order $d$ of the finite group $\Pic^0(C)[n]$ of $n$-torsion line bundles over $C$ and 
let  $\AAA^0_\Gamma \subset \AAA^0_n$ denote the
endoscopic sublocus of characteristics $a$ such that the associated degree $n$ spectral cover
$\pi: X_a \ra C$ comes from a degree $\frac{n}{d}$ spectral cover over the \'etale Galois
cover of $C$ with Galois group $\Gamma$ (for the precise definition see section \ref{sectcyclicGalois}). 
With this notation we have the following

\begin{thm}
We have an equivalence
$$ \Gamma \subset K_a \qquad \iff \qquad a \in \AAA^0_\Gamma. $$
\end{thm}

\bigskip

This gives a description of the locus of characteristics $a \in \AAA^0_n$ such that the Prym
variety $\Prym(X_a/C)$ is non-connected, because clearly $\AAA^0_{\Gamma_2} \subset
\AAA^0_{\Gamma_1}$ if $\Gamma_1 \subset \Gamma_2$.

\begin{cor}
The sublocus of characteristics $a \in \AAA_n^0$ such that the Prym variety $\Prym(X_a/C)$ is not 
connected equals the union
\begin{equation} \label{aendo}
\AAA_{endo} := \bigcup \AAA^0_\Gamma,
\end{equation}
where $\Gamma$ varies over all cyclic subgroups of prime order of $\Pic^0(C)[n]$.
\end{cor}

Calculating the dimensions of the endoscopic loci $\AAA^0_\Gamma$ will lead to an immediate 
topological application. Recall that $\Pic^0(C)[n]$ acts on $\calM$ by tensorization,
and this will induce an action on the rational cohomology $H^*(\calM;\QQ)$. We then have

\begin{thm}\label{variant} Let $n>1$ and $p_n$ be the smallest prime  divisor of $n$. Assume that $M=K_C$, the
canonical bundle of $C$, and that $(n,\deg(\Delta))=1$. Then the action of $\Pic^0(C)[n]$ on $H^k(\calM;\QQ)$ is trivial, provided that
 $$k\leq 2n^2(1-1/p_n)(g-1).$$ \end{thm}

In fact this result should be sharp, as the topological mirror symmetry conjecture \cite[Conjecture 5.1]{hausel-thaddeus} predicts that the smallest degree where 
$\Pic^0(C)[n]$ acts non-trivially is $$k=n^2(1-1/p_n)(2g-2)+1.$$ This results hints at the close connection between Ng\^{o}'s strategy in \cite{ngo1,ngo2} for studying the symmetries of the Hitchin fibration and the topological mirror symmetry conjectures in \cite{hausel-thaddeus}. More discussion on this   connection can be found in \cite{hausel-survey}.

Finally let $\calN$ denote the moduli space of stable vector bundles of rank $n$ and fixed determinant $\Delta$ over 
$C$. Again the finite group $\Pic^0(C)[n]$ acts on  $\calN$ by tensorization and thus on $H^*(\calN;\QQ)$. As the cohomology $H^*(\calN;\QQ)$   is a summand in the cohomology of $H^*(\calM;\QQ)$ we can deduce

\begin{cor}\label{hnmain} The finite group $\Pic^0(C)[n]$ acts trivially on $H^*(\calN;\QQ)$. 
\end{cor}

This was the  main application of Harder--Narasimhan in \cite[Theorem 1]{harder-narasimhan}.  Our proof here can be considered as an example of both  the abelianization philosophy of Atiyah--Hitchin \cite[\S 6.3]{atiyah}  and Ng\^{o}'s strategy \cite{ngo1,ngo2} of studying the symmetries of the Hitchin fibration.

\bigskip

The paper is organized as follows. In sections 2 and 3 we recall basic results on spectral covers
and on the norm map $\Nm_{X/C}$. In sections 4 and 5 we prove the two main theorems. In section 6 we 
describe the action of the Prym variety $\Prym(X_a/C)$ on the 
fiber over $a \in \AAA^0_n$ of the $\SL_n$-Hitchin
fibration. Finally in section 7 we apply the results in this paper to
prove Theorem~\ref{variant} and its Corollary~\ref{hnmain}.

\bigskip

{\em Notation:} Given a sheaf $\cF$ over a scheme $X$ and a subset $U \subset X$ we denote by $\cF(U)$ or 
by $\Gamma(U,\cF)$ the space of sections of $\cF$ over $U$.

\bigskip

{\em Acknowledgements:} During the preparation of this paper the first author was supported by a Royal Society University Research Fellowship 
while the second author was supported by a Marie Curie
Intra-European Fellowship (PIEF-GA-2009-235098) from the European Commission. We would like to thank Jean-Marc Dr\'ezet,  Nigel Hitchin, M.S. Narasimhan and Daniel Schaub for helpful discussions.

\section{Preliminaries}

\subsection{Two lemmas on abelian varieties}

Given an abelian variety $A$ and a positive integer $n$ we denote by $[n]: A \rightarrow A$ the 
multiplication by $n$, by $A[n] = \ker [n]$ its subgroup of $n$-torsion points and by 
$\hat{A} = \Pic^0(A)$ its dual abelian variety. We consider 
$$  f : A \longrightarrow B $$
a homomorphism between abelian varieties with kernel $K = \ker(f)$ which we 
assume to be finite. We let 
$\hat{f} : \hat{B} \ra \hat{A}$ denote the dual map induced by $f$. We introduce 
the quotient abelian variety $A' = A/K$, so that we can 
write the homomorphism $f$ as a composite map
$$ f = j \circ \mu : A \map{\mu} A' \map{j} B,$$
where $\mu$ is an isogeny with kernel $K$ and $j$ is injective.

\begin{lem} \label{mapf}
The group of connected components of the abelian subvariety $\ker(\hat{f}) \subset \hat{B}$ equals
$$ \pi_0( \ker(\hat{f})) = \widehat{K}, $$
where $\widehat{K} = \Hom(K, \CC^*)$ is the group of characters of $K$.
\end{lem}

\begin{proof}
We consider the dual map
$$\hat{f} : \hat{B} \map{\hat{j}} \hat{A'} \map{\hat{\mu}} \hat{A},$$
and observe that $\hat{\mu} : \hat{A'} \ra \hat{A}$ is an isogeny with kernel $\widehat{K}$
(see e.g. \cite{BL} Proposition 2.4.3) and $\hat{j}$ has connected fibers (see e.g.
\cite{BL} Proposition 2.4.2). The lemma then follows. 
\end{proof}

We also suppose that $A$ and $B$ are principally polarized abelian varieties, i.e. the 
polarizations induce isomorphisms $A \cong \hat{A}$ and $B \cong \hat{B}$. 

\begin{lem} \label{lemmatwo}
We assume that there exists a homomorphism $g : B \ra A$ such that $g \circ f = [n]$ for some integer $n$. 
Then the dual of the canonical
inclusion $i : K \hookrightarrow A[n]$ is a surjective map 
$$ \hat{i} : A[n] = \hat{A}[n] \longrightarrow \widehat{K},$$
which coincides with the restriction to $A[n]$ of the composite map
$\hat{j} \circ \hat{g} : A \ra B \ra \hat{A'}$.
\end{lem}

\begin{proof}
It suffices to observe that the isogeny $\hat{f} \circ \hat{g} = \widehat{[n]} = [n]$ factorizes as
$$ [n] : A \map{\hat{j} \circ \hat{g}} \hat{A'} \map{\hat{\mu}} A, $$
that $\widehat{K} = \ker( \hat{\mu} )$, and that $\hat{j} \circ \hat{g}$ is surjective.
Hence a canonical surjection $A[n] \ra \widehat{K}$, which is dual to the inclusion
$i: K \hookrightarrow A[n]$, since $\widehat{[n]} = [n]$.
\end{proof}

\bigskip

\subsection{Spectral covers} \label{sectionspectralcovers}

In this section we review some elementary facts on spectral covers.

\bigskip

Let $C$ be a complex smooth projective curve and let $M$ be a line bundle over $C$ with $\deg M > 0$. We denote
by $|M|$ the total space of $M$ and by 
$$ \pi : |M| \lra C $$
the projection onto $C$. There is a canonical coordinate $t \in H^0(|M|, \pi^*M)$ on the total space
$|M|$. The direct image decomposes as follows
$$ \pi_* \cO_{|M|} = \Sym^\bullet(M^{-1}) = \bigoplus_{i=0}^\infty M^{-i}. $$

\begin{defi}
A spectral cover $X$ of degree $n$ over the curve $C$ and associated to the line bundle $M$ is the
zero divisor in $|M|$ of a non-zero section $s \in H^0(|M|, \pi^* M^{n})$.
\end{defi}

Since a spectral cover $X$ is a subscheme of $|M|$, it is naturally equipped with a projection onto $C$,
which we also denote by $\pi$. The decomposition of  the section $s$ according to the direct sum
\begin{eqnarray*}
H^0(|M|, \pi^* M^{n}) & =  & H^0(C, M^n \otimes \bigoplus_{i=0}^\infty M^{-i}) \ \ \ \ \text{(projection formula)} \\
& = & H^0(C, M^n) \oplus \cdots \oplus H^0(C, M) \oplus H^0(C, \cO_C)
\end{eqnarray*}
gives an expression
$s = s_0 + ts_{1} + \cdots + t^{n-1} s_{n-1} + t^n s_n$ with $s_j \in H^0(C, M^{n-j})$. Here we
also denote by $s_j$ its pull-back to $|M|$. We note that there is an isomorphism
$\pi^* : \Pic(C) \stackrel{\sim}{\longrightarrow} \Pic(|M|)$, hence any line bundle over $|M|$ is of the
form $\pi^* L$ for some line bundle $L \in \Pic(C)$. More generally, we have a decomposition
$H^0(|M|, \pi^* L) = H^0(C, L) \oplus H^0(C, LM^{-1}) \oplus \cdots \oplus H^0(C,LM^{-d})$ for 
some integer $d$ and any
section $s \in H^0(|M|, \pi^* L)$  can be written in the form
\begin{equation}  \label{decequation}
s = s_0 + ts_{1} + \cdots + t^{d-1} s_{d-1} + t^d s_d, \qquad s_j \in H^0(C, LM^{-j}).
\end{equation}

\begin{lem} \label{irrcomspecur}
Let $\pi : X \ra C$ be a spectral cover. Then the underlying
reduced curve of each irreducible component of $X$ is again a spectral cover associated to the line
bundle $M$. 
\end{lem}

\begin{proof}
It suffices to show that if the section $s \in H^0(|M|, \pi^* M^n)$ decomposes as
$s = s^{(1)} \cdot s^{(2)}$ with $s^{(i)} \in H^0(|M|, \pi^* L_i)$ for $i=1,2$ and
$L_1 L_2 = M^n$, then $L_i = M^{n_i}$ and $n_1 + n_2 = n$. By \eqref{decequation} the
section $s^{(i)}$ can be written as 
\begin{equation} \label{decequation1}
s^{(i)} = s^{(i)}_0 + ts^{(i)}_{1} + \cdots + t^{n_i} s^{(i)}_{n_i} ,
\end{equation}
with $s_j^{(i)} \in  H^0(C, L_i M^{-j})$ and $s_{n_i}^{(i)}
\not= 0$. Moreover $n_i = \deg (X^{(i)}/C)$ with $X^{(i)} = \mathrm{Zeros}(s^{(i)})$. By considering the
highest order terms of \eqref{decequation1} we obtain the relations $n_1 + n_2 = n$ and 
$s^{(1)}_{n_1} \cdot s^{(2)}_{n_2} = s_n \in H^0(C, \cO_C)$. Since $s_n$ is a non-zero 
constant section, we conclude that $L_i = M^{n_i}$. 
\end{proof}

We introduce the $\SL_n$- and $\GL_n$-Hitchin space for the line bundle $M$ over the curve $C$
$$\AAA^0_n(C,M) = \bigoplus_{j=2}^n H^0(C,M^j) \ \ \ \text{and} \ \ \ \AAA_n(C,M) = \bigoplus_{j=1}^n H^0(C,M^j). $$
If no confusion arises, we simply denote these vector spaces by $\AAA^0_n$ and $\AAA_n$. Note that
$\AAA^0_n \subset \AAA_n$.
Given an element $a = (a_1, \ldots , a_n) \in \AAA_n$ with $a_j \in H^0(C, M^j)$, called a
characteristic, we associate to $a$ a spectral cover of degree $n$
$$ \pi_a : X_a \lra C, \qquad X_a = \mathrm{Zeros}(s_a) \subset |M|, $$
with
\begin{equation} \label{polynomiala}
s_a = t^n + a_1 t^{n-1} + \cdots a_{n-1} t + a_n \in H^0(|M|, \pi^* M^n).
\end{equation}

\begin{rem} \label{translationa1}
{\em Given $a \in \AAA_n$ we observe that the pull-back of the spectral cover $X_a \subset |M|$ by
the automorphism of $|M|$ given by translation with the section $- \frac{a_1}{n}$, i.e. $(x,y) 
\mapsto (x, y - \frac{1}{n}a_1(x))$, equals the spectral cover $X_{a'}$ for some $a' \in \AAA^0_n$;
equivalently do the change of variables $t \mapsto t - \frac{a_1}{n}$. Hence $X_a \cong X_{a'}$. It
therefore suffices to restrict our study to spectral covers $X_a$ for $a \in \AAA_n^0$. }
\end{rem}

\bigskip




\subsection{Non-reduced curves}

Let $X$ be an irreducible curve contained in a smooth surface and let $X^{red}$ denote its underlying
reduced curve. Then there exists a global section $s$ of a line bundle such that $X^{red} = \mathrm{Zeros}(s)$ and
an integer $k$ such that $X = \mathrm{Zeros}(s^k)$. We introduce the 
subschemes $X_i = \mathrm{Zeros}(s^i)$ for $i = 1, \ldots, k$,
so that we have a filtration of $X$ by closed subschemes
$$ X^{red} = X_1 \subset X_2 \subset \cdots \subset X_k = X. $$
In that case we say that $X$ has a nilpotent structure of order $k$. For any integer $i$ we denote by
$\cO_{X_i}$ the structure sheaf of the subscheme $X_i \subset X$. Note that $\cO_{X_i}$ is naturally
a $\cO_X$-module.

\bigskip

We need to recall a result on the local structure of coherent sheaves on non-reduced curves.

\begin{thm}[\cite{D} Th\'eor\`eme 3.4.1] \label{strthmnonreduced}
Let $X$ be a curve with nilpotent structure of order $k$ and let $\cE$ be a coherent sheaf over $X$. Then
there exists an open subset $V \subset X$ depending on $\cE$  and integers $m_i$ such that
$$ \cE_{|V} \stackrel{\sim}\lra  \bigoplus_{i= 1}^k {\cO_{X_i}^{\oplus m_i}}_{|V}.$$
The sheaf on the right is called a quasi-free sheaf.
\end{thm}

\bigskip

\section{The norm map}

In this section we recall the definition of the norm map and prove some of its properties. The
standard references are \cite{EGA2} section 6.5 and \cite{EGA4} section 21.5.

\subsection{Definition} \label{defnorm}

Let $C$ be a smooth projective curve and let 
$$ \pi : X \longrightarrow C $$
be any finite degree $n$ covering of $C$. The $\cO_C$-algebra $\pi_* \cO_C$ will be denoted $\cB$ and the
group of invertible  elements in $\cB$ by $\cB^*$. Note that $\cB$ is a locally free sheaf of rank $n$. 
Let $U \subset C$ be an open subset and let $s \in \Gamma(U, \cB) = \Gamma( \pi^{-1}(U) , \cO_X)$ be a
local section. One defines (\cite{EGA2} section 6.5.1)
$$ N_{X/C} (s) := \det (\mu_s) \in \Gamma(U, \cO_C) $$
where $\mu_s : \cB_{|U} \ra  \cB_{|U}$ is the multiplication with the section $s$. Moreover $s$ is
invertible in $\Gamma(U, \cB)$ if and only if $N_{X/C} (s)$ is invertible in $\Gamma(U, \cO_X)$. 
We have the following obvious relations
\begin{equation} \label{propnorm}
 N_{X/C}(s \cdot s') = N_{X/C}(s) \cdot N_{X/C}(s'), \qquad N_{X/C}( \lambda s ) = \lambda^n 
N_{X/C}(s)
\end{equation}
for any local sections $s$ and $s'$ of $\cB$ and any local section $\lambda$ of $\cO_C$.

\bigskip

Let $\cL$ be an invertible $\cB$-module. We can choose a covering $\{ U_\lambda \}_{\lambda \in \Lambda}$
of $C$ by open subsets  and trivializations $\eta_\lambda : \cL_{| U_\lambda } \map{\sim}  
\cB_{| U_\lambda }$. Then $( \omega_{\lambda,\mu} )_{\lambda,\mu \in \Lambda}$ with 
$$ \omega_{\lambda,\mu}  = \eta_\lambda \circ \eta^{-1}_{\mu|U_\lambda \cap U_\mu} \in \Gamma(U_\lambda
\cap U_\mu, \cB)$$
is a $1$-cocycle with values in $\cB^*$ and $(N_{X/C}(\omega_{\lambda, \mu}))_{\lambda,\mu \in \Lambda}$
is a $1$-cocycle with values in $\cO_C^*$, the sheaf of invertible elements of $\cO_C$. This 
$1$-cocycle determines an invertible sheaf over $C$, which we denote by $\Nm_{X/C}(\cL)$. The
following properties easily follow from \eqref{propnorm}
\begin{equation} \label{Nmpullback}
 \Nm_{X/C}(\cL \otimes \cL') = \Nm_{X/C}(\cL) \otimes \Nm_{X/C}(\cL'), \qquad
  \Nm_{X/C}(\pi^* \cM) = \cM^{\otimes n}, 
\end{equation}
for any two invertible sheaves $\cL$ and $\cL'$ over $X$ and for any invertible sheaf $\cM$ over $C$.
We therefore obtain a group homomorphism between the Picard
groups of the curves $X$ and $C$ called the norm map
$$\Nm_{X/C} : \Pic(X) \lra \Pic(C), \qquad \cL \mapsto \Nm_{X/C}(\cL). $$

\bigskip

\subsection{Properties}

In the case $X$ is smooth, the norm map $\Nm_{X/C}$ has a more explicit description in terms of divisors 
associated to line bundles.

\begin{prop}[\cite{EGA4} section 21.5]
Assume that $X$ is a smooth curve. The norm map, as defined above, coincides with the map
$$ \cL = \cO_X( \sum_{i \in I} n_i p_i ) \mapsto  \Nm_{X/C} (\cL) = \cO_C ( \sum_{i \in I} n_i \pi(p_i) ), $$
where $n_i \in \ZZ$ and $p_i \in X$. Note that this map is well-defined, i.e.  $\Nm_{X/C} (\cL)$
only depends on the linear equivalence class of the divisor  $\sum_{i \in I} n_i p_i$.
\end{prop}

\bigskip

From now on the curve $X$ is again an arbitrary cover of $C$.

\begin{lem} 
Let $0 \ra \cE \ra \cF \ra \cT \ra 0$ be an exact sequence of $\cO_X$-modules. We assume that
$\cE$ and $\cF$ are torsion-free and that $\cT$ is a torsion sheaf. Let $\varphi_\bullet$ be a local
morphism over $\pi^{-1}(U)$ for some open subset $U \subset C$ between exact sequences
\begin{equation}  \label{morvarphi}
\begin{array}{ccccccc}
0 \lra & \cE & \lra & \cF & \lra & \cT & \lra 0 \\
       & \ \ \downarrow^{\varphi_\cE} & & \  \ \downarrow^{\varphi_\cF} & & \ \ \downarrow^{\varphi_\cT} &  \\
 0 \lra & \cE & \lra & \cF & \lra & \cT & \lra 0.
\end{array}
\end{equation}
We consider the $\cO_C$-linear maps
induced by $\varphi_\cE$ and $\varphi_\cF$ in the direct image sheaves
$\pi_* \cE$ and $\pi_* \cF$.
Then we have the equality
$$\det (\varphi_\cE) = \det (\varphi_\cF)  \in \Gamma(U, \cO_U).$$
\end{lem}

\begin{proof}
It is enough to show that the two local sections $\det (\varphi_\cE)$ and $\det (\varphi_\cF)$ coincide
in the local rings $\cO_{C,p}$ for every point $p \in U$. We put $A = \cO_{C,p}$ and $K = Fr(A)$ and denote by 
$E$, $F$ and $T$ the corresponding $A$-modules of sheaves $\cE$, $\cF$ and $\cT$. Then $E$ and $F$ are free
$A$-modules, hence we have injections $E \hookrightarrow E \otimes_A K$ and 
$F \hookrightarrow F \otimes_A K$. Since $T$ is a torsion module, we have $T \otimes_A K = 0$. Then after 
localizing \eqref{morvarphi} at $p \in C$ and taking tensor product with $K$, we obtain the commutative
diagram
$$
\begin{CD}
E \otimes_A K @>\cong>> F \otimes_A K \\
@VV \varphi_E \otimes id V   @VV \varphi_F \otimes id V \\
E \otimes_A K @>\cong>> F \otimes_A K, 
\end{CD}
$$
where the horizontal maps are isomorphisms. So $\varphi_E \otimes id$ and $\varphi_F \otimes id$ are 
conjugate, hence $\det (\varphi_E \otimes id) = \det (\varphi_F \otimes id) \in K$. On the other hand
$\det (\varphi_E \otimes id)$ and $\det (\varphi_F \otimes id)$ are elements in $A \subset K$, hence we obtain the desired
equality.
\end{proof}

In the sequel we will use the following properties of the norm map:

\bigskip

\begin{cor} \label{equalitydet}
Let $\cE$ and $\cF$ be two torsion-free $\cO_X$-modules such that
$$ 0 \lra \cE \lra \cF \lra \cT \lra 0, $$
where $\cT$ is a torsion $\cO_X$-module. Let $s \in \Gamma(U, \cB) = \Gamma(\pi^{-1}(U), \cO_X)$
be a local section of $\cB$ over the open subset $U \subset C$. We consider the maps
induced by the multiplication with the section $s$ in the direct image sheaves
$\pi_* \cE$ and $\pi_* \cF$, which we denote by $\mu_s^\cE \in \Hom_{\cO_C(U)}(\pi_* \cE(U), \pi_* \cE(U))$
and  $\mu_s^\cF \in \Hom_{\cO_C(U)}(\pi_* \cF(U), \pi_* \cF(U))$. Then we have the equality
$$ \det(\mu_s^\cE ) = \det(\mu_s^\cF ) \in \Gamma( U, \cO_C). $$
\end{cor}

\begin{lem} \label{Nmpullbackp}
Let $p : \widetilde{X} \ra X$ be a covering such that the cokernel of the canonical inclusion
$\cO_X \hookrightarrow p_*  \cO_{\widetilde{X}}$ is a torsion $\cO_X$-module. Then, for any
invertible sheaf $\cL$ over $X$ we have
$$   \Nm_{\widetilde{X}/C} (p^* \cL)  = \Nm_{X/C} (\cL). $$ 
\end{lem}

\begin{proof}
We consider the exact sequence
\begin{equation} \label{exactseqtorsion}
  0 \lra \cO_X \lra p_* \cO_{\widetilde{X}} \lra \cT \lra 0, 
\end{equation}
where $\cT$ is a torsion $\cO_X$-module. Note that the direct image
$ p_* \cO_{\widetilde{X}} $ is torsion-free. We denote the $\cO_C$-algebra
$\pi_* p_* \cO_{\widetilde{X}}$ by $\widetilde{\cB}$. Note that $\widetilde{\cB}$ is
a $\cB$-module. Let $\cL$ be an invertible $\cO_X$-module, $\eta_\lambda : \cL_{|U_\lambda}
\map{\sim}  \cB_{|U_\lambda}$ be a set of trivializations of $\cL$ as 
$\cB$-module, and $( \omega_{\lambda,\mu} )_{\lambda,\mu \in \Lambda}$ be the
corresponding $1$-cocycle with values in $\cB^*$. Then the pull-back $p^* \cL$
corresponds to a $1$-cocycle  $( p^* \omega_{\lambda,\mu} )_{\lambda,\mu \in \Lambda}$ 
with values in $\widetilde{\cB}^*$ obtained from $( \omega_{\lambda,\mu} )_{\lambda,\mu \in \Lambda}$
under the canonical inclusion $\cB \hookrightarrow \widetilde{\cB}$. We now apply Corollary 
\ref{equalitydet} to the exact sequence \eqref{exactseqtorsion} and conclude that
$N_{\widetilde{X}/C} (p^* \omega_{\lambda,\mu} )  = N_{X/C} (\omega_{\lambda,\mu}) \in
\Gamma( U_\lambda \cap U_\mu , \cO_C)$. This proves the lemma.
\end{proof}

\begin{lem} \label{Nmred}
Let $X = \bigcup_{i=1}^r X_i$ be the decomposition of $X$ into irreducible components $X_i$. For an
invertible sheaf $\cL$, we denote by $\cL_i = \cL \otimes_{\cO_X} \cO_{X_i}$ its restriction to $X_i$.
Then, we have the equality
$$ \Nm_{X/C}(\cL) = \bigotimes_{i=1}^r \Nm_{X_i/C}(\cL_i). $$
\end{lem}

\begin{proof}
We apply the previous lemma to the covering $p : \widetilde{X} = \bigsqcup_{i=1}^r X_i \ra X$ given by the
disjoint union of the curves $X_i$.
\end{proof}

\begin{lem} \label{Nmnr}
Let $X$ be an irreducible curve and let $j: X^{red} \hookrightarrow X$ be its underlying reduced curve. Let
$m$ be the multiplicity of $X^{red}$ in $X$. Then, for any invertible sheaf $\cL$ over $X$ we have
$$ \Nm_{X/C} (\cL) = \Nm_{X^{red}/C} (j^* \cL)^{\otimes m}. $$ 
\end{lem}

\begin{proof}
The $\cO_C$-algebra $\cB = \pi_* \cO_X$ comes equipped with a nilpotent ideal sheaf $\cJ \subset \cB$ such
that $\cB_{red} = \cB/ \cJ = \pi_* \cO_{X^{red}}$. We choose a covering $\{ U_\lambda \}_{\lambda \in \Lambda}$
of $C$ by open subsets which trivialize the invertible sheaf $\cL$, i.e. there 
exists isomorphisms $\eta_\lambda : \cL_{| U_\lambda } \map{\sim}  \cB_{| U_\lambda }$ and  such that $\cJ_{|U_\lambda}$
is generated by an element $t \in \cB_{| U_\lambda}$. Then multiplication with the invertible element 
$\omega_{\lambda,\mu} =  \eta_\lambda \circ \eta^{-1}_{\mu|U_\lambda \cap U_\mu}$ preserves the 
filtration $t^{m-1} \cB_{| U_\lambda } \subset \cdots \subset t 
\cB_{| U_\lambda } \subset \cB_{| U_\lambda }$ and acts on  the quotients as multiplication with
$\omega_{\lambda,\mu}^{red} \in \cB^{red}_{| U_\lambda \cap U_\mu}$. It follows that $N_{X/C}( \omega_{\lambda,\mu}) = 
N_{X^{red}/C}(\omega_{\lambda,\mu}^{red})^{m}$, which proves the lemma. 
\end{proof}

\subsection{The Prym variety $\Prym(X/C)$}

Given a spectral cover $\pi: X \ra C$ we denote by $\Pic^0(X)$ the connected component of the identity
element of the Picard group of $X$ (see e.g. \cite{K}). We then define the Prym variety $\Prym(X/C)$ to be 
the kernel of the Norm map $\Nm_{X/C}$
$$ \Prym(X/C) := \ker \left( \Nm_{X/C} : \Pic^0(X) \lra \Pic^0(C)  \right). $$

We recall that $n$ denotes the degree of the cover $\pi: X \ra C$. We choose an ample line bundle $\cO_C(1)$ over $C$ and
denote by $\cO_X(1) = \pi^* \cO_C(1)$ its pull-back to $X$ and by $\delta = \deg \cO_C(1)$.

\begin{defi} \label{definitionsemistability}
Let $\cE$ be a coherent $\cO_X$-module. The rank and degree of $\cE$ with respect to the polarization
$\cO_X(1)$ are the rational numbers $\rk (\cE)$ and $\deg (\cE)$ determined by the Hilbert polynomial
$$  \chi(X, \cE \otimes \cO_X(l)) = n \delta l \rk (\cE) + \deg (\cE) + \rk (\cE) \chi (\cO_X). $$
The slope of $\cE$ is defined by $\mu(\cE) = \frac{\deg(\cE)}{\rk(\cE)}$.
The sheaf $\cE$ is stable (resp. semi-stable) if $\cE$ is torsion-free and for any proper subsheaf 
$\cE' \subset \cE$ we have the equality $\mu(\cE') < \mu(\cE)$ (resp. $\leq$).
\end{defi}

\begin{rem}
{\em The definitions of rank and degree of a coherent sheaf $\cE$ over $X$ above  coincide with the classical ones
when the curve $X$ is integral. The (semi-)stability condition above coincides with the (semi-)stability condition 
introduced in \cite{Simpson1}.}
\end{rem}

\begin{rem}
{\em Using the equality $\chi(X, \cE \otimes \cO_X(l)) = \chi(C,\pi_* \cE \otimes \cO_C(l))$ we obtain the 
following formulae
\begin{equation} \label{degreedirectimage}
n \rk(\cE)  = \rk (\pi_* \cE) \qquad \text{and} \qquad \deg(\cE) + \rk(\cE) \chi(\cO_X) = 
\deg(\pi_* \cE) + \rk(\pi_* \cE) \chi(\cO_C).
\end{equation}
}
\end{rem}

\begin{prop} \label{detpi}
Let $\cE$ be a torsion-free $\cO_X$-module of integral rank $r = \rk (\cE)$ and let $\cL$ be an invertible $\cO_X$-module. Then we have the relation
$$ \det ( \pi_* (\cE \otimes \cL)) = \det ( \pi_* \cE ) \otimes \Nm_{X/C}(\cL)^{\otimes r}. $$
\end{prop}

\begin{proof}
We shall use the notation of section \ref{defnorm}. Since $\cE$ is torsion-free, the direct image
$\pi_* \cE$ is a locally free $\cO_C$-module. We choose a covering $\{ U_\lambda \}_{\lambda \in \Lambda}$ of
$C$ for which both $\cL$ and $\pi_* \cE$ are trivialized, i.e., such that there exists local
isomorphisms
$$ \alpha_\lambda : \pi_* \cE_{|U_\lambda} \stackrel{\sim}{\lra} \cO_{U_\lambda}^{\oplus rn}, \qquad
\tau_\lambda : \cL_{|U_\lambda} \stackrel{\sim}{\lra} \cB_{U_\lambda} . $$
Since $\cL$ is trivial on $U_\lambda$ we have an isomorphism
$$ \mathrm{id}_\cE \otimes \tau_\lambda : \cE \otimes \cL_{|U_\lambda} \lra  \cE \otimes \cB_{|U_\lambda}, $$
which we can consider as an isomorphism between $\cO_C$-modules
$$   \mathrm{id}_\cE \otimes \tau_\lambda : \pi_* (\cE \otimes \cL)_{|U_\lambda} \lra  \pi_*  \cE_{|U_\lambda}. $$
We compose with $\alpha_\lambda$ to obtain a trivialization of $\pi_* (\cE \otimes \cL)_{|U_\lambda}$
$$ \beta_\lambda : \pi_* (\cE \otimes \cL)_{|U_\lambda}  \stackrel{\mathrm{id}_\cE \otimes \tau_\lambda}{\lra}
 \pi_*  \cE_{|U_\lambda} \stackrel{\alpha_\lambda}{\lra} \cO_{U_\lambda}^{\oplus rn}. $$
Given $\lambda, \mu \in \Lambda$ we can now write the transition functions $f_{\lambda,\mu} = 
\beta_\lambda \circ \beta_\mu^{-1}$ of the vector
bundle $\pi_*(\cE \otimes \cL)$ as
$$f_{\lambda,\mu} : \cO_{U_\lambda}^{\oplus rn} \stackrel{\alpha_\mu^{-1}}{\lra} 
(\pi_*  \cE)_{|U_{\lambda, \mu}}   \stackrel{\mathrm{id}_\cE \otimes \omega_{\lambda, \mu}}{\lra}
(\pi_*  \cE)_{|U_{\lambda, \mu}}  \stackrel{\alpha_\lambda}{\lra}  \cO_{U_\lambda}^{\oplus rn}, $$
where we denote by $\omega_{\lambda, \mu} = \tau_\lambda \circ \tau_\mu^{-1}$ the 
$\cB^*$-valued transition functions of the line bundle $\cL$. We deduce from this expression  the
relation
$$ \det (f_{\lambda, \mu}) = \det (g_{\lambda, \mu}) \cdot \det( \mathrm{id}_\cE \otimes \omega_{\lambda, \mu}), $$
where $g_{\lambda, \mu} = \alpha_\lambda \circ \alpha_\mu^{-1}$ denotes the transition functions of the 
vector  bundle $\pi_* \cE$. Hence the proposition follows if we show the relation 
$\det( \mathrm{id}_\cE \otimes \omega_{\lambda, \mu}) = \det(\omega_{\lambda, \mu})^r$, which is proved 
in the next Lemma.
\end{proof}

\begin{lem}
Let $\cE$ be a torsion-free $\cO_X$-module and let $s \in \Gamma(U, \cB) = \Gamma(\pi^{-1}(U), \cO_X)$ be
a local section of $\cB$ over the open subset $U \subset C$. We denote by 
$\mu_s^\cE \in \Hom_{\cO_C(U)}(\pi_* \cE(U), \pi_* \cE(U))$ the map induced by multiplication
with the section $s$. Then we have an equality
$$ \det (\mu_s^\cE) = \det (\mu_s) ^r \in \Gamma(U, \cO_C). $$ 
\end{lem}

\begin{proof}
By Lemma \ref{strthmnonreduced} there exists an open subset $j: V \hookrightarrow X$ such that 
$j^* \cE$ is isomorphic to $j^* \cQ$ where $\cQ$ is a quasi-free sheaf of the form 
$\oplus_{i=1}^k \cO_{X_i}^{\oplus m_i}$. We then apply Corollary \ref{equalitydet} to the two exact sequences
$$ 0 \lra \cE \lra j_* j^* \cE \lra \cT_1 \lra 0, \qquad \text{and} \qquad
0 \lra \cQ \lra j_* j^* \cQ \lra \cT_2 \lra 0, $$
where $\cT_i$ are torsion sheaves. This leads to the equality $\det(\mu_s^\cE) = \det(\mu_s^\cQ)$. It therefore
suffices to compute $\det(\mu_s^\cQ)$ in terms of $\det(\mu_s)$. We put $n = k \cdot l$ with $l = \deg(X^{red}/C)$.
Then we have 
$$ r = \rk(\cE) =  \rk(\cQ) = \frac{1}{n} \sum_{i=1}^k m_i \rk(\pi_* \cO_{X_i}) = \frac{1}{n} \sum_{i=1}^k m_i i l
= \frac{1}{k} \sum_{i=1}^k m_i i.$$
Let $A = \cO_{C,p}$ denote the local ring at the point $p \in C$ and let $B$ denote the localization
of $\pi_* \cO_X$ at the point $p \in C$. Thus $B$ is a projective $A$-module of rank $n$ equipped with a 
filtration
$$ t^{k-1} B \subset \cdots \subset tB \subset B, \qquad t \in B \ \text{with} \ t^k = 0. $$
We put $B_1 = B/tB$, the localization of $\pi_* \cO_{X^{red}}$ at the point $p \in C$. Since $B$ is 
projective we can choose a splitting 
$$ B = B_1 \oplus t B_1 \oplus \cdots \oplus t^{k-1} B_1. $$
Using this decomposition we can write a section $s \in B$ as $s = s_0 + t s_1 + \cdots + t^{k-1} s_{k-1}$
with $s_j \in B_1$. Moreover, the localization of $\pi_* \cO_{X_i}$ at the point $p \in C$ is given by
$B_i :=  B_1 \oplus t B_1 \oplus \cdots \oplus t^{i-1} B_1$ and the matrix of the multiplication with $s$ 
in $B_i$ is with respect to this decomposition lower block-triangular and has determinant 
$\det( \mu_s^{B_i}) = \det (\mu_{s_0}^{B_1})^i$. 
Therefore 
$$ \det(\mu_s^\cQ) = \prod_{i=1}^k \det(\mu_s^{B_i})^{m_i} = \det (\mu_{s_0}^{B_1})^{\sum_{i=1}^k im_i}. $$
On the other hand $\det(\mu_s) = \det (\mu_s^{\cO_X}) = \det( \mu_s^{B_k}) = \det (\mu_{s_0}^{B_1})^k$,
which leads to the desired equality.
\end{proof}

Taking the trivial sheaf $\cE = \cO_X$ in Proposition \ref{detpi} we obtain the following description of the 
norm map:

\begin{cor}
For any invertible $\cO_X$-module, we have 
$$\Nm_{X/C}(\cL) = \det(\pi_* \cL) \otimes \det(\pi_* \cO_X)^{-1}.$$
\end{cor}

\bigskip

\section{The group of connected components of $\Prym(X/C)$}

In this section we give the proof of Theorem 1.1.

\subsection{Part (1)}

Given a spectral cover $X$ we will associate a covering 
$$ p : \widetilde{X} \lra X $$
as follows: let $X = \bigcup_{i = 1}^r X_i$ be its decomposition
into irreducible components $X_i$,  let $X_i^{red}$ be the underlying reduced curve of 
$X_i$, let $m_i$ be the multiplicity of $X_i^{red}$ in $X_i$ and let $\widetilde{X}_i^{red}$ be the normalization
of $X_i^{red}$. Since $X_i^{red}$ is embedded in the smooth surface $|M|$, there exists a 
sequence of blowing-ups $bl : \widetilde{|M|} \ra  |M|$ of the surface $|M|$ at reduced points (depending
on the curve $X_i^{red}$) such that
the proper transform of $X_i^{red}$ equals its normalization $\widetilde{X}_i^{red}$. We define
$\widetilde{X}_i \subset \widetilde{|M|}$  to be the proper transform of the non-reduced curve $X_i \subset |M|$. 
We take
$$\widetilde{X} = \bigsqcup_{i=1}^r \widetilde{X}_i$$ 
to be the disjoint union of the curves $\widetilde{X}_i$ together
with the natural map $p$ onto $X$. Note that the multiplicity of $\widetilde{X}_i^{red}$ in $\widetilde{X}_i$ also
equals $m_i$.

\begin{lem} \label{normnonred} 
The covering $p: \widetilde{X} \ra X$ constructed above  has the following properties:
\begin{enumerate}
\item the cokernel of the canonical inclusion $\cO_X \hookrightarrow p_* \cO_{\widetilde{X}}$ is a torsion
$\cO_X$-module,
\item the underlying reduced curve $\widetilde{X}^{red}$ of $\widetilde{X}$ is smooth.
\item the map induced by pull-back under $p$
$$ \Pic^0(X) \map{p^*} \Pic^0(\widetilde{X})$$
is surjective and has connected kernel.
\item we have an equality
$$ \pi_0(\Prym (X/C)) = \pi_0(\Prym (\widetilde{X}/C)).$$
\end{enumerate}
\end{lem}

\begin{proof}
\begin{enumerate}
\item This is clear since $p : \widetilde{X} \ra X$ is an isomorphism outside a finite set of points.

\item We clearly have $\widetilde{X}^{red} = \bigsqcup_{i=1}^r \widetilde{X}^{red}_i$ and the
curves $\widetilde{X}^{red}_i$ are smooth by constuction.

\item We consider the two exact sequences obtained by restricting invertible sheaves to the underlying reduced curve
\begin{equation*}
\begin{array}{ccccccc}
0 \lra & U_1 & \lra & \Pic(X) & \lra & \Pic(X^{red}) & \lra 0 \\
    &  \downarrow^\alpha &  & \downarrow^{p^*} &    & \downarrow^{p_{red}^*} &  \\
0 \lra & U_2 & \lra & \Pic(\widetilde{X}) & \lra & \Pic(\widetilde{X}^{red}) & \lra 0,
\end{array}
\end{equation*}
which are surjective with unipotent kernels $U_1$ and $U_2$ by \cite{Liu} Lemma 7.5.11. Then by the
snake lemma the kernel $\ker(p^*)$ fits into the exact sequence
$$ 0 \lra \ker(\alpha) \lra  \ker(p^*) \lra \ker (p_{red}^*) \lra \coker (\alpha) \lra 0.$$
Note that $\ker(\alpha)$ and $\coker(\alpha)$ are unipotent groups.
We shall denote by $V$ the kernel of the last map. By \cite{Liu} Lemma 7.5.13 the kernel $\ker( p_{red}^*)$ is an
extension of a toric group by an unipotent group. The same holds for $V$, since there
are no non-zero maps from a toric group to an unipotent group. Hence $V$ and $\ker(\alpha)$  are connected, so
$\ker (p^*)$ is connected. Hence $\ker (p^*)$ is contained in the connected component $\Pic^0(X)$ and we
obtain that $p^* : \Pic^0(X) \ra \Pic^0(\widetilde{X})$ is surjective.

\item
Because of Lemma \ref{Nmpullbackp} we have an exact sequence
$$ 0 \lra \ker(p^*) \lra \Prym(X/C) \map{p^*} \Prym(\widetilde{X}/C) \lra 0.$$
The equality between the groups of connected components now follows since $\ker(p^*)$ is
connected.
\end{enumerate}
\end{proof}

The previous lemma implies that it is enough to show the equality $\pi_0(\Prym (\widetilde{X}/C)) = \widehat{K}$.
By Lemma \ref{Nmred} and Lemma \ref{Nmnr} the Norm map $\Nm_{\widetilde{X}/C}$
factorizes as follows
$$ \Nm_{\widetilde{X}/C} : \Pic^0(\widetilde{X}) \map{j^*} \Pic^0(\widetilde{X}^{red}) 
= \prod_{i=1}^r \Pic^0(\widetilde{X}^{red}_i) \map{\prod [m_i]}  
\prod_{i=1}^r \Pic^0(\widetilde{X}^{red}_i) \map{\prod \Nm } \Pic^0(C). $$
Moreover $j^*$ is surjective and $\ker (j^*)$ is connected (see e.g.\cite{Liu} Lemma 7.5.11). It suffices therefore to compute
$\pi_0( \ker (h))$, where $h : \Pic^0(\widetilde{X}^{red})  \ra  \Pic^0(C)$ denotes the 
composite of the last two maps. We also consider the composite homomorphism
$$ f : \Pic^0(C) \map{\Delta} \Pic^0(C)^r \map{\prod [m_i]} \Pic^0(C)^r
 \map{\prod \widetilde{\pi}_i^*} 
\prod_{i=1}^r \Pic^0(\widetilde{X}_i^{red}) = \Pic^0(\widetilde{X}^{red}),$$
where $\Delta(L) = (L,\ldots, L)$ is the diagonal map.
We note that the duals $\widehat{\widetilde{\pi}}_i^*$ and $\widehat{[m_i]}$ coincide 
with $\Nm_{\widetilde{X}_i^{red}/C}$ and $[m_i]$ under the identifications
$\widehat{\Pic^0(C)} \cong \Pic^0(C)$ and $\widehat{\Pic^0(\widetilde{X}_i^{red})} \cong 
\Pic^0(\widetilde{X}_i^{red})$ 
given by the principal 
polarizations on the Jacobians (see \cite{BL} section 11.4), and that the dual $\widehat{\Delta}$ of
$\Delta$ is the multiplication map
on $\Pic^0(C)$ (see e.g. \cite{BL} exercise 2.6 (12)). Hence we obtain that $\hat{f} = h$. Thus
$\pi_0(\Prym (\widetilde{X}/C)) = \pi_0 (\ker(\hat{f}))$. Now we apply Lemma \ref{mapf} to $f$ 
and we obtain the desired result since $\ker(f) =  \bigcap_{i=1}^r(K_i)_{m_i}$.

\subsection{Part (2)}

We consider the morphism $f : \Pic^0(C) \ra \Pic^0(\widetilde{X}^{red})$ introduced in the previous section. Moreover
the morphism $g : \Pic^0(\widetilde{X}^{red}) \ra \Pic^0(C)$ defined by 
$$g(\cL_1, \ldots, \cL_r) = \bigotimes_{i=1}^r \Nm_{\widetilde{X}^{red}_i/C}(\cL_i)$$ 
satisfies the relation
$g \circ f = [n]$. We are therefore in a position to apply Lemma \ref{lemmatwo} to the morphism $f$. This proves
part (2) for the Prym variety $\Prym(\widetilde{X}/C)$. Since  by Lemma \ref{normnonred} the natural map $p^* :
\Pic^0(X) \ra \Pic^0(\widetilde{X})$ induces an isomorphism $\pi_0(\Prym(X/C)) = \pi_0 (\Prym(\widetilde{X}/C))$,
we are done.

\subsection{Part (3)}

The if part follows immediately from the formula proved in part (1). Suppose now that $K = \Pic^0(C)[n]$.
With the  notation above we have $n = \sum_{i = 1}^r  m_i \deg (X_i^{red}/ C)$ and $K = \bigcap_{i=1}^r (K_i)_{m_i}$, 
from which we deduce that $r = 1$.
On the other hand $K = (K_1)_{m_1} = \Pic^0(C)[n]$ implies that $K_1 = \Pic^0(C)[d_1]$ with 
$d_1= \deg ( X_1^{red}/ C)$. But this can only
happen if $d_1 = 1$. Hence $m_1 = n$ and we are done. 

\bigskip

\section{Endoscopic subloci of $\AAA_n$}

\subsection{Cyclic Galois covers} \label{sectcyclicGalois}

We consider a smooth projective curve $C$ and a line bundle $M \in \Pic(C)$. Let $\Gamma$ be a cyclic
subgroup of order $d$ of the group of $n$-torsion line bundles $\Pic^0(C)[n]$ and let 
$$ \varphi : D \lra C $$
be the \'etale Galois covering of $C$ associated  to $\Gamma \subset \Pic^0(C)[n]$. By definition
$D = \Spec( \cE_\Gamma)$ where $\cE_\Gamma = \oplus_{L \in \Gamma} L$ is the direct sum of
all line bundles  $L$ in $\Gamma$ with the natural $\cO_C$-algebra structure. Note that the Galois
group of the covering $\varphi : D \ra C$ equals $\Gamma \cong \Aut(D/C)$. We 
introduce the line bundle $N = \varphi^* M$. Then the line bundle $N$ has a canonical 
$\Gamma$-linearization, hence we obtain a canonical action of $\Gamma$ on the total space
$|N|$. We notice that the canonical coordinate $t \in H^0(|N|, \pi^* N)$ is invariant 
under this $\Gamma$-action.

\bigskip

We consider a spectral cover of degree $m$ over $D$ with
associated line bundle $N$ given by a global section $s \in H^0(|N|, \pi^* N^m)$. We 
can apply a Galois automorphism $\sigma \in \Gamma$ to $s$ and denote its image by
$s^\sigma$. We introduce 
$$ \widehat{s} = \prod_{\sigma \in \Gamma} s^\sigma \in H^0(|N|, \pi^* N^n), \qquad \text{with}  \ \  n = d \cdot m. $$
We observe that $\widehat{s}$ is $\Gamma$-invariant, hence $\widehat{s}$ descends to a section
over $|M|$, which we also denote by $\widehat{s}$. Hence we obtain a map
$$ \Phi_\Gamma : \AAA_m(D,N) \lra \AAA_n := \AAA_n(C,M),  \qquad b \mapsto \Phi_\Gamma(b). $$
with $a = \Phi_\Gamma(b)$ defined by the relation $\widehat{s}_b = s_a$, where $s_b \in
H^0(|N|, \pi^* N^m)$ is the global section $s_b = t^m + b_1 t^{m-1} + b_2 t^{m-2} + \cdots + b_m$
associated to $b = (b_1, \ldots, b_m)$ with $b_j \in H^0(D, N^j)$. Since the zero divisor of the section $\widehat{s}$
has a finite number of irreducible components, we immediately see that the fiber $\Phi_\Gamma^{-1}(\widehat{s})$ is
finite, hence $\dim  \ \im \Phi_\Gamma = \dim \AAA_m(D,N)$.  We also introduce the subspace
$$ \AAA_m^{\Gamma}(D,N) = H^0(D, N)_{var} \oplus  \bigoplus_{j=2}^m H^0(D,N^j) \subset \AAA_m(D,N), $$
where 
$H^0(D, N)_{var}$ denotes the $\Gamma$-variant
subspace of $H^0(D, N)$, i.e. the direct sum of the character spaces $H^0(D, N)_\chi$ for non-trivial
characters $\chi$ of the group $\Gamma$.

\begin{lem}
We have the inclusion
$$ \Phi_\Gamma (\AAA_m^{\Gamma}(D,N)) \subset \AAA^0_n. $$ 
\end{lem}

\begin{proof}
It suffices to compute the coefficient of $t^{n-1}$ in $\widehat{s}_b $, which equals 
$\sum_{\sigma \in \Gamma} \sigma^*b_1$. We immediately see that the relation 
$\sum_{\sigma \in \Gamma} \sigma^*b_1 = 0$ is equivalent to $b^{(0)}_1 = 0$, where 
$b^{(0)}_1$ denotes the $\Gamma$-invariant component of $b_1$.
\end{proof}

We denote the images of $\Phi_\Gamma$ by 
$$\AAA_\Gamma^0 \subset \AAA_n^0 \qquad \text{and} \qquad \AAA_\Gamma \subset \AAA_n.$$ 
The subvariety $\AAA_\Gamma$
admits the following characterization: for $a \in \AAA_n$ we denote by 
$$Y_a = X_a \times_C D$$ 
the fiber product of $X_a$ and $D$ over $C$. Then $Y_a$ is a spectral cover over $D$ of degree $n$ associated to the
line bundle $N$. The following lemma follows immediately from the definition of $\AAA_\Gamma$.

\begin{lem} \label{charAGamma}
The characteristic $a$ lies in $\AAA_\Gamma$ if ond only if the fiber product $Y_a$ decomposes as 
$$Y_a = \bigcup_{\sigma \in \Gamma} Z^\sigma,$$
where $Z$ is a spectral cover of degree $m = \frac{n}{d}$ over $D$ and $Z^\sigma$ is its image under the Galois
automorphism $\sigma \in \Gamma$.
\end{lem}

\bigskip

We also need to introduce some natural subvarieties of the Hitchin spaces $\AAA^0_n$ and 
$\AAA_n$, which will be used in the proof of Theorem 1.2.

\bigskip

For any divisor $l \not= 1$ of $n$, with $n = k \cdot l$, we consider the natural
$k$-th power map
$$ \Phi_k : \AAA_l \lra \AAA_n, $$
where $\Phi_k(b) = a$ is defined by the relation
$$ s_a = (t^l + b_1 t^{l-1} + \cdots + b_l)^k \in H^0(|M|, \pi^*M^n), \qquad \text{for} \
b = (b_1, \ldots, b_l) \in \AAA_l. $$
We shall abuse notation and will also denote by $\AAA_l$ its image $\Phi_k(\AAA_l) \subset \AAA_n$.
Note that $\Phi_k(\AAA^0_l) \subset \AAA^0_n$ and we also denote this image by $\AAA^0_l$.

\bigskip

Given two positive integers $n_1,n_2$ such that $n_1 + n_2 = n$, we introduce the map
$$ \Phi_{n_1,n_2} : \AAA_{n_1} \times \AAA_{n_2} \lra \AAA_n, $$
with $a = \Phi_{n_1,n_2}(b,c)$ defined by the relation $s_a = s_b \cdot s_c$, where
$s_b = t^{n_1} + b_1 t^{n_1 -1} + \cdots + b_{n_1}$ and $s_c = t^{n_2} + c_1 t^{n_2 -1} + \cdots + c_{n_2}$ 
for $b = (b_1, \ldots, b_{n_1}) \in \AAA_{n_1}$ and $c = (c_1, \ldots, c_{n_2}) \in \AAA_{n_2}$. We define $(\AAA_{n_1} \times  \AAA_{n_2})_0 \subset 
\AAA_{n_1} \times \AAA_{n_2}$ to be the
subset of pairs $(b,c)$ satisfying the relation $b_1 + c_1 = 0$. We shall denote by 
$\AAA_{n_1, n_2} \subset \AAA_n$ the image of $\Phi_{n_1,n_2}$ and by  $\AAA^0_{n_1, n_2}$ the
subset $\AAA_{n_1, n_2} \cap 
\AAA^0_n = \Phi_{n_1,n_2}[(\AAA_{n_1} \times  \AAA_{n_2})_0 ]$. 
 
\bigskip

\subsection{Proof of Theorem 1.2}

We show here the analogue of Theorem 1.2 for the $\GL(n)$-Hitchin space $\AAA_n$. Note that both 
statements are equivalent by Remark \ref{translationa1}.
Given a spectral cover $\pi: X_a \ra C$ with $a \in \AAA_n$, we denote the subgroup of $\Pic^0(C)$ defined
in \eqref{defK} by $K_a$. Let $\Gamma \subset \Pic^0(C)[n]$ be a cyclic subgroup of order $d$.

\begin{thm}
We have an equivalence
$$ \Gamma \subset K_a \qquad \iff \qquad a \in \AAA_\Gamma. $$
\end{thm}

\begin{proof}
We first show the equivalence in the case the spectral cover $X_a$ is integral. In that case
we can consider its normalization $\widetilde{X}_a$, which comes with a natural projection
$$ \widetilde{\pi}_a : \widetilde{X}_a \lra C.$$
By \cite{BL} Proposition 11.4.3 we have $\Gamma \subset K_a$ if and only if  $\widetilde{\pi}_a$ factors
through the map $\varphi$, i.e. there exists a map $u: \widetilde{X}_a \ra D$ such that 
$\widetilde{\pi}_a = \varphi \circ u$. By the universal property of the fiber product there exists
a map $\delta: \widetilde{X}_a \ra Y_a$ into the fiber product $Y_a$ of $X_a$ with $D$ over $C$.
We denote by $Z = \im (\delta) \subset Y_a$ the image of the smooth irreducible curve $\widetilde{X}_a$.
Then $Z$ is irreducible too. Moreover, since $X_a$ is reduced and $\varphi$ is \'etale, the curve
$Y_a$ is also reduced, hence $Z$ is integral. The group $\Gamma$ acts on $Y_a$, hence permutes
its irreducible components. Since $\Gamma$ acts transitively on the fibers of $Y_a \ra X_a$, all
irreducible components are of the form $Z^\sigma$ for some $\sigma \in \Gamma$. We therefore obtain a factorization
$\widetilde{X}_a  \ra Z \ra X_a$. Since this composite map is birational, we deduce that
$\deg(Z/X_a) = 1$. Hence, since $\deg(Y_a/X_a) = d$, we conclude that
$$ Y_a = \bigcup_{\sigma \in \Gamma} Z^\sigma \qquad \text{and} \qquad Z^\sigma \not= Z^{\sigma'} \ \
\text{if} \ \sigma \not= \sigma'. $$
By Lemma \ref{irrcomspecur} the curve $Z$ is a spectral cover of degree $m$ over $D$ and by Lemma \ref{charAGamma}
we obtain that $a \in \AAA_\Gamma$.

\bigskip

Conversely, for $a \in \AAA_\Gamma$ the map $Z \ra X_a$ given by Lemma \ref{charAGamma} is birational. Hence
the normalization of $Z$ equals $\widetilde{X}_a$ and we obtain a factorization $\widetilde{X}_a \ra Z \ra D \ra C$,
which implies that $\Gamma \subset K_a$ by \cite{BL} Proposition 11.4.3.

\bigskip

Now we will prove the equivalence for more general characteristics $a \in \AAA_n$. We start with 
$a \in \AAA_n$ such that the spectral cover $X_a$ is irreducible, but not reduced. 
Let $X_a^{red}$ be the underlying reduced curve of $X_a$ and let $k$ be the 
multiplicity of $X_a^{red}$ in $X_a$. We put $n = k \cdot l$. By Lemma \ref{irrcomspecur} we have 
$X_a^{red} = X_{a_{red}}$ for some characteristic $a_{red} \in \AAA_l \subset \AAA_n$ and $a =
\Phi_k(a_{red})$ --- see section \ref{sectionspectralcovers}. Then by formula \eqref{defK} we have
$K_a = [k]^{-1} (K_{a_{red}})$. We introduce $\Gamma_{red} = [k](\Gamma) \subset \Pic^0(C)[l]$. Then
$\Gamma_{red}$ is a cyclic subgroup of order $d_{red} = \frac{d}{\mathrm{gcd}(k,d)}$. With this notation
we easily obtain the equivalence
$$ \Gamma \subset K_a \qquad \iff \qquad \Gamma_{red} \subset K_{a_{red}}. $$
We combine this equivalence  with the statement of the Theorem written for the integral characteristic $a_{red}$,
which was proved above:
$$ \Gamma_{red} \subset K_{a_{red}} \qquad \iff \qquad a_{red} \in \AAA_{\Gamma_{red}}. $$
Therefore it remains to show the following equivalence
$$ a_{red} \in \AAA_{\Gamma_{red}} \qquad \iff \qquad a = \Phi_k(a_{red}) \in \AAA_\Gamma. $$
In order to show this equivalence we introduce the subgroup  
$S = \ker (\Gamma \ra \Gamma_{red})$. By Galois theory there exists an intermediate cover
$D \ra \overline{D} \ra C$ with $\Aut(\overline{D}/C) = \Gamma_{red}$ and
$\Aut(D/\overline{D}) = S$.

\bigskip

Consider a characteristic $a_{red} \in \AAA_{\Gamma_{red}}$. By Lemma \ref{charAGamma}
applied to $a_{red} \in \AAA_{\Gamma_{red}}$ we obtain that the fiber product 
$Y_{a_{red}} = X_{a_{red}} \times_C \overline{D}$ decomposes as  
$\bigcup_{\sigma \in \Gamma_{red}} W^\sigma$, where $W$ is a spectral cover of degree $\frac{l}{d_{red}}$ 
over $\overline{D}$. Now,
observing that $Y_a = k (Y_{a_{red}}  \times_{\overline{D}} D)$ as divisors in $|N|$, we
can write
$$ Y_a = k \bigcup_{\sigma \in \Gamma_{red}}(W \times_{\overline{D}} D)^\sigma = 
\bigcup_{\sigma \in \Gamma} Z^\sigma, $$
where we have put $Z = \frac{k}{\mathrm{gcd}(k,d)} (W \times_{\overline{D}} D) \subset |N|$. Note that
$Z^\sigma = Z$ for $\sigma \in S$ and that $Z$ is a spectral cover of degree $\frac{n}{d}$. 
This proves that $a = \Phi_k(a_{red}) \in \AAA_\Gamma$.

\bigskip

Conversely, we consider a characteristic $a_{red} \in \AAA_l$ with $\Phi_k(a_{red}) \in \AAA_\Gamma$. We assume
that the spectral cover $X_{a_{red}}$ is integral. This assumption implies that the fiber product
$X_{a_{red}} \times_C D$ is reduced. Let $\cI$ denote an irreducible component of $X_{a_{red}} \times_C D$,
let $\mathrm{Stab}(\cI)$ denote its stabilizer, i.e.,
$$ \mathrm{Stab}(\cI) = \{ \sigma \in \Gamma \ | \ \cI^\sigma = \cI \}, $$
and let $\delta = | \mathrm{Stab}(\cI) |$ be the order. Since $\Gamma$ acts transitively on the fibers of
$X_{a_{red}} \times_C D \ra X_{a_{red}}$ we obtain the decomposition into 
irreducible components 
$$ X_{a_{red}} \times_C D  = \bigcup_{\sigma \in \Gamma / \mathrm{Stab}(\cI)} \cI^\sigma. $$
Let us denote by $s$ the global section over $|N|$ with $\mathrm{Zeros}(s) = \cI$. Then the
spectral cover $Y_a = k (X_{a_{red}} \times_C D)$ is the zero set of the section
$$  \prod_{\sigma \in \Gamma / \mathrm{Stab}(\cI)} (s^\sigma )^k, $$
which has $k \frac{d}{\delta}$ irreducible factors of the same degree. The
assumption $\Phi_k(a_{red}) \in \AAA_\Gamma$ implies that this product can be written as
a product of $d$ factors of the same degree, hence $k \frac{d}{\delta}$ is divisible by $d$, i.e., 
$\delta$ divides $k$, so $\delta$ divides $\mathrm{gcd}(k,d)$. Since $\delta = | \mathrm{Stab}(\cI) |$ and
$\mathrm{gcd}(k,d) = |S|$, we conclude that $\mathrm{Stab}(\cI) \subset S$. We then introduce the
section 
$$t = \prod_{\sigma \in S / \mathrm{Stab}(\cI)} s^\sigma.$$ 
Since $t$ is $S$-invariant, its
zero divisor descends as a spectral cover $W$ over $\overline{D} = D / S$. Moreover we have the
equality
$$ \prod_{\sigma  \in \Gamma / \mathrm{Stab}(\cI)} s^\sigma = \prod_{\tau  \in \Gamma_{red} = \Gamma/ S} t^\tau, $$
which proves that $Y_{a_{red}} =  X_{a_{red}} \times_C \overline{D}$ splits into $d_{red}$ spectral
covers $W^\tau$ for $\tau \in \Gamma_{red}$, and we conclude by Lemma \ref{charAGamma} that 
$a_{red} \in \AAA_{\Gamma_{red}}$.

\bigskip

Finally, we will show the equivalence for a characteristic $a \in \AAA_{n_1,n_2} \subset \AAA_n$, i.e. the
spectral cover $X_a$ equals the union $X_{a_1} \cup X_{a_2}$ for two spectral covers $X_{a_i}$ with 
$a_i \in \AAA_{n_i}$, which
we assume to be irreducible. Then by \eqref{defK} we have
$$ \Gamma \subset K_a \qquad \iff \qquad \Gamma \subset K_{a_1} \ \text{and} \ \Gamma \subset K_{a_2}. $$
On the other hand since the curves $X_{a_i}$ are irreducible, we can apply what we have proved above, i.e., for
$i= 1,2$
$$ \Gamma \subset K_{a_i}  \qquad \iff \qquad a_i \in \AAA^{n_i}_\Gamma. $$
Note that $\Gamma \subset \Pic^0(C)[n_i]$ for $i= 1,2$. Here $\AAA^{n_i}_\Gamma$ denotes the corresponding
subspace of $\AAA_{n_i}$. Hence it remains to show that
$$ a_1 \in \AAA^{n_1}_\Gamma \ \text{and} \  a_2 \in \AAA^{n_2}_\Gamma  \qquad \iff \qquad
a = \Phi_{n_1,n_2}(a_1,a_2) \in \AAA_\Gamma.$$
The implication $\Rightarrow$ is trivial. In order to show the implication $\Leftarrow$ we note that Lemma
\ref{charAGamma} gives the decomposition $Y_a = \bigcup_{\sigma \in \Gamma} Z^\sigma$ for some spectral cover
$Z$. We then put $Z_i = Z \cap Y_{a_i}$, which gives the desired decomposition for the fiber product
$Y_{a_i}$.

\bigskip

Now the statement follows for arbitrary characteristic $a \in \AAA_n$ by induction on the
number of irreducible components of $X_a$.

\end{proof}
 
\section{Moduli space of semi-stable Higgs bundles}

In this section we describe how the Prym variety $\Prym(X_a/C)$ acts on the fiber $h^{-1}(a)$ of the
$\SL_n$-Hitchin fibration. First we recall the semi-stability condition \ref{definitionsemistability}
for coherent sheaves over the spectral cover $X_a$. The following result generalizes
Proposition 3.6 of \cite{BNR} to any spectral cover $X_a$.

\begin{prop} 
Let $a \in \AAA_n^0$ be any characteristic.
The fiber $h^{-1}(a)$ of the $\SL_n$-Hitchin fibration \eqref{Hitchinfibration} equals the moduli space of
semi-stable sheaves $\cE$ over the spectral curve
$X_a$ of rank $1$ and with fixed determinant $\det (\pi_* \cE) = \Delta$. 
\end{prop} 

\begin{proof}
In the case $X_a$ integral this is exactly \cite{BNR} Proposition 3.6 except for the assertion that the
bijective correspondence between Higgs bundles $(E,\phi)$ over $C$ with given characteristic $a$  and
torsion free rank-$1$ sheaves $\cE$ over $X_a$ is compatible with both semi-stability conditions. 
But this compatibility is shown in \cite{Simpson2} Corollary 6.9 --- note that Simpson works in the special case 
when $M$ is the canonical bundle of $C$ but his proof remains valid for an arbitray line bundle $M$.

\bigskip

In order to show the statement for an arbitrary spectral cover $X_a$ we observe that the constructions and arguments of \cite{BNR} and \cite{Simpson2} remain valid for an arbitrary characteristic $a$. However the next statement needs
a new proof.

\begin{lem}
Let $\cE$ be a torsion-free sheaf over $X_a$. Then the characteristic of its associated Higgs bundle $(E,\phi)$
equals $a$.
\end{lem}

\begin{proof}
If $X_a$ is integral, this is shown in \cite{BNR} Proposition 3.6 using the irreducibility of the polynomial $s_a$ defined
in \eqref{polynomiala}. In the general case we first show the statement when $X_a$ is irreducible but non-reduced and secondly when $X_a$ is reducible. Given a sheaf $\cE$ over $X_a$ we denote by $\mathrm{Char}(\cE)$ the characteristic
polynomial of its associated Higgs field. 

Suppose first that $X_a$ is irreducible and that its underlying reduced curve $X_a^{red}$ has multiplicity $m$. Denote
by $\cE^{red}$ the restriction of $\cE$ to $X_a^{red}$. Then one shows exactly along the lines of Lemma \ref{Nmred} that
$\mathrm{Char}(\cE) = \mathrm{Char}(\cE^{red})^m$.

Suppose now that $X_a$ has several irreducible components $X_i$ and denote by $\cE_i$ the restriction of $\cE$ to
$X_i$. Then one shows exactly along the lines of Lemma \ref{Nmnr} that  
$\mathrm{Char}(\cE) =  \prod_i \mathrm{Char}(\cE_i)$.

Finally decomposing $X_a$ into irreducible components and taking the underlying reduced curves, one reduces the
statement to the integral case, which is already shown.
\end{proof}

This completes the proof of the proposition.
\end{proof}

Considering $h^{-1}(a)$ as the moduli space of semi-stable sheaves over $X_a$ we let $\Prym(X_a/C)$ act through
tensor product. By Proposition \ref{detpi} it is clear that the determinant is invariant under this action, and since
$\deg(\cE \otimes \cL) = \deg (\cE)$ for any $\cL \in \Prym(X_a/C)$ we see that semi-stability is preserved.

\begin{rem}
{\em In a forthcoming paper we will describe in detail the case $a=0$, i.e. the action of the
Prym variety $\Prym(X_0/C)$, which equals in this case $n^{2g}$ copies of a vector space (see e.g. \cite{Drezet} section 3.3), on the nilpotent cone of the Higgs moduli space.
}
\end{rem}

\begin{rem}
{\em We note that the  description of the fiber $h^{-1}(a)$ above was already stated in Proposition 2.1 in \cite{Schaub}.
Unfortunately the proof of the Proposition 2.1 in \cite{Schaub} contains an inaccuracy, which is a consequence of the 
author's different definition of a torsion-free sheaf of rank $1$ over an arbitrary spectral cover. }
\end{rem}

\section{Application to Topological Mirror Symmetry}

In this final section we assume that our line bundle $M$ equals the canonical bundle $K_C$ and that 
$(n,{\rm deg} (\Delta))=1$. These assumptions imply that $\calM$ is a non-singular quasi-projective variety of dimension $$\dim(\calM)=(n^2-1)(2g-2).$$ The dimension of the affine space $\AAA_n^0$ is $$\dim(\AAA_n^0)=(n^2-1)(g-1)$$ and consequently the Hitchin map $h$ is of relative dimension $(n^2-1)(g-1)$.  

Let $\Gamma\subset \Pic^0(C)[n]$ be 
a cyclic subgroup of order $d$ which must divide $n$. Then we have
\begin{lem} \label{dim} $\dim(\AAA^0_\Gamma)=(n^2/d-1)(g-1)$
\end{lem}
\begin{proof} In \S\ref{sectcyclicGalois} we associated an \'etale Galois cover $\phi:D\to C$ to $\Gamma$ with Galois group $\Gamma$. Thus the pull back $N=\phi^*(K_C)=K_D$. There we also constructed $\AAA^0_\Gamma$ as the image of $\AAA^\Gamma_m$ by the finite map $\Phi_\Gamma$. 
 Thus we have \begin{multline*}\dim(\AAA^0_\Gamma)=\dim(\AAA^\Gamma_m)=
\dim(\AAA_m)-\dim(H^0(D,K_D)^\Gamma)=\dim(\AAA_m)-\dim(H^0(C,K_C))=\\ =\frac{n^2}{d^2}(dg-d)+1-g=({n^2}/d-1)(g-1).\end{multline*}
\end{proof}
Now we can prove Theorem~\ref{variant}
\begin{proof} With the notation \eqref{aendo} we introduce  
$$\AAA^0_{ne}:=\AAA_n^0\setminus \AAA_{endo}$$ 
the locus of those characteristics for which $\Prym(X_a/C)$ is connected. Further denote 
$$\calM_{ne}:=h^{-1}(\AAA^0_{ne} ).$$ 
First we argue that $\Pic^0(C)[n]$ acts
trivially on $H^*(\calM_{ne})$.   Let $a\in \AAA^0_{ne}$. Since by \eqref{Nmpullback} we have 
$\pi^* \Pic^0(C)[n] \subset \Prym(X_a/C)$ and since by assumption on the characteristic $a$ the Prym
variety $\Prym(X_a/C)$ is connected, the finite group $\Pic^0(C)[n]$ acts trivially on the cohomology 
of every fiber $h_{ne}^{-1}(a)$ of $$h_{ne}:=h|_{\calM_{ne}}: \calM_{ne}\to \AAA^0_{ne}.$$
Therefore by Lemma~\ref{fibretriv} below it follows that it acts trivially on $H^*(\calM_{ne};\QQ)$ as well. 

\begin{lem}\label{fibretriv} Let $f:X\to Y$ be a proper map between locally compact Hausdorff topological spaces. 
Let the finite group $\G$ act on $X$ along the fibers of $f$. Assume $\G$ acts trivially on $H^*(X_a;\QQ)$ for every $a\in Y$, where $X_a=f^{-1}(a)$. Then the action of $\G$ on $H^*(X;\QQ)$ is trivial.
\end{lem}
\begin{proof} Let $\underline{\QQ}_X$ be the constant sheaf on $X$. Let $\pi:X\to X/\G$ be the quotient map, $\underline{\QQ}_{X/\G}$ the constant sheaf on $X/\G$ and $g:X/\G\to Y$ defined by the property $g\circ \pi=f$. Then we have a morphism of sheaves $$m:\underline{\QQ}_{X/\G} \to \pi_*\underline{\QQ}_X.$$ This induces a morphism of sheaves $$R^km:R^k g_*\underline{\QQ}_{X/\G}\to R^kg_*\pi_*\underline{\QQ}_X.$$ As $\pi$ is finite $R^kg_*\pi_*\underline{\QQ}_X\cong R^kf_*\underline{\QQ}_X$. Thus we get a morphism of sheaves 
$$R^km:R^k g_*\underline{\QQ}_{X/\G}\to R^kf_*\underline{\QQ}_X.$$ Firs we understand the induced map on any stalk. By proper base change \cite[Theorem 6.2]{iversen} the stalks are just the $k$th cohomology groups of the fibres. Thus our morphism becomes $$R^km_s: H^k(X_s/\G;\QQ)\to H^k(X_s;\QQ).$$ We have  the isomorphism $$H^k(X_s/\G;\QQ)\cong H^k(X_s;\QQ)^\G$$ of \cite[(1.2)]{macdonald} and by assumption  $$H^k(X_s;\QQ)^\G\cong H^k(X_s;\QQ).$$ Thus $R^km_s$ is an isomorphism for all $s\in Y$  so \cite[Proposition 1.1]{hartshorne} implies that  $R^km$ is an isomorphism of sheaves. By \cite[(1.2)]{macdonald} the theorem follows.   \end{proof}
  
  Thus we have that $\Pic^0(C)[n]$ acts trivially on $H^*(\calM_{ne};\QQ)$.
Note that by Lemma~\ref{dim} and by the observation that $\AAA_{\Gamma_2}^0 \subset \AAA_{\Gamma_1}^0$ if
$\Gamma_1 \subset \Gamma_2$, the codimension of the endoscopic locus $\AAA_{endo}$ is given by 
$$c_n:=(n^2-1)(g-1)-({n^2}/p_n-1)(g-1)=n^2(1-1/p_n)(g-1),$$ 
where $p_n$ is the smallest prime divisor of $n$.  By studying the cohomology long exact sequence of the pair $(\calM,\calM_{ne})$ we see that the restriction map $$H^k(\calM;\QQ)\to H^k(\calM_{ne};\QQ)$$ is an isomorphism for $k\leq 2c_n-2$ and is an injection for $k=2c_n-1$.Thus we could immediately deduce that $\Pic^0(C)[n]$ acts trivially on $H^k(\calM;\QQ)$ for \begin{eqnarray} k< 2c_n.\label{elem}
 \end{eqnarray} Further, we note that  any generic $(c_n-1)$-dimensional subvariety $\Lambda^{c_n-1}$  of $\AAA^0$  will be disjoint from $\cup \AAA^0_\Gamma$ thus a  cohomology class $\eta 
\in H^*(\calM;\QQ)$ which is not invariant under $\Pic^0(C)[n]$  (we call such classes {\em variant}) must satisfy $$\eta|_{h^{-1}(
\Lambda^{c_n-1})}
=0$$ and so by \cite[Theorem 1.4.8]{dCHM} a variant class of degree $i$ must have perversity at most ${i-c_n}$. By  
\cite[Theorem 1.4.12] {dCHM} this already implies $i\geq 2c_n$ (which as we noted above in \eqref{elem}  also follows from the cohomology  long exact sequence of $(\calM,\calM_{ne})$). Finally a variant class $$\eta\in H^{2c_n}_{c_n}(\calM;\QQ)$$  by Relative Hard Lefschetz -with the ample $\Pic^0(C)[n]$-invariant class $\alpha \in H^2(\calM;\QQ)$-  gives a variant class $$\alpha^{\dim(\calM)/2-c_n}\eta \in H^{\dim(\calM)}(\calM;\QQ)$$ which contradicts \cite[Theorem 1]{GHS}, which proves that there are no variant classes in the middle cohomology $H^{\dim(\calM)}(\calM;\QQ)$. Here we used the fact $\alpha\in H^2(\calM;\QQ)$ is always $\Pic^0(C)[n]$-invariant, for example because of \eqref{elem} and 
 $2<2c_n$ when $g>1$.

  Thus the smallest possible degree 
a variant class may have is $2c_n+1$ proving Theorem~\ref{variant}.
\end{proof}

\begin{rem} {\em Theorem~\ref{variant}  is a consequence of the topological mirror symmetry conjecture
\cite[Conjecture 5.1]{hausel-thaddeus}. It can be deduced by calculating all possible fermionic shifts $F(\gamma)$ for $\gamma\in \Pic^0(C)[n]$ on the RHS of \cite[Conjecture 5.1]{hausel-thaddeus} and noting that due to the presence of the gerbe $\hat{B}^d$  the degree $0$ part of the cohomology of the twisted sector $\calM^\gamma/\Pic^0(C)[n]$ does not contribute to the stringy cohomology of $\calM/\Pic^0(C)[n]$. In fact, the inequality in Theorem~\ref{variant} is sharp, when $n$ is a prime, because \cite[Proposition 10.1]{hausel-thaddeus} implies that there is variant cohomology in degree 
$2c_n+1$. The close connection between the topological mirror symmetry 
conjecture \cite[Conjecture 5.1]{hausel-thaddeus} and the group of connected components of  Prym varieties of spectral covers unravelled in this section is an indication of the deep analogies between S-duality considerations in physics as in Kapustin-Witten's \cite{kapustin-witten} and Ng\^{o}'s \cite{ngo1,ngo2} geometric approach to the fundamental lemma. For more discussion on this see \cite{hausel-survey}.}
\end{rem}

Finally we can prove Corollary~\ref{hnmain} 
\begin{proof} As the Hitchin map $h:\calM\to \AAA_n^0$ is proper by \cite{nitsure}, and the 
$\CC^\times$-action on $\calM$ covers a $\CC^\times$ action on $\AAA_n^0$ with
positive weights, and so with a unique fixed point $0\in \AAA^0$. It follows that for 
every $z\in \calM$ the limit $\lim_{\lambda \to 0} \lambda z$ exists in the compact fixed point set $\calM^{\CC^\times}$. Using the compactification $\overline{\calM}$ as defined in such a situation in \cite[\S 11]{Simpson3} and 
\cite{hausel-compact} we can 
apply the following cohomological techniques. They were studied  in the compact case by Kirwan in \cite{kirwan},
(who commented on the non-compact case in \cite[\S 9]{kirwan})
 and in the non-compact case by Nakajima in \cite[\S 5.1]{nakajima}.  

 For a connected component $F_\alpha$ of the fixed point set $$\calM^{\CC^\times}=\bigsqcup_{\alpha\in I} F_\alpha$$ we can then define $$U_\alpha:=\{z\in \calM | \lim_{\lambda \to 0} \lambda z \in F_\alpha\}$$ giving the Bialinicki-Birula decomposition $$\calM=\bigsqcup_{\alpha\in I} U_\alpha.$$ One can now define
a partial ordering on the index set $I$, by $\alpha\leq \beta$ when $U_\beta \cap \overline{U}_\alpha \neq\emptyset$. By induction on this ordering one can prove \begin{eqnarray}\label{decomp}H^*(\calM;\QQ)\cong \bigoplus_{\alpha \in I} H^{*-{\rm codim}(U_\alpha)}(U_\alpha;\QQ).\end{eqnarray} If $0$ denotes the minimal element in $I$, then we can see that $U_0=T^*\calN$ contains those Higgs bundles where the underlying bundle is stable. It follows from  \eqref{decomp} (or rather via the perfectness of the Bialinicki-Birula stratification used to prove \eqref{decomp}) that the restriction map $$H^*(\calM;\QQ)\to H^*(U_0;\QQ)\cong H^*(\calN;\QQ)$$ 
is surjective. 

Thus in particular Theorem~\ref{variant} implies that $\Pic^0(C)[n]$ acts trivially on the cohomology $H^k(\calN;\QQ)$ for $k\leq 2c_n=2n^2(1-1/{p_n})(g-1)$, where $p_n$ is the smallest prime divisor of $n$. But $$
2c_n>(n^2-1)(g-1)$$ thus the cohomology $H^*(\calN;\QQ)$ is $\Pic^0(C)[n]$-invariant up to and including the middle degree  and  by Poincar\'e duality everywhere.     The result follows.

As $2c_n-1\geq (n^2-1)(g-1)$ (with possible equality when $g=2$ and $p_n=2$) we note that the more elementary estimate \eqref{elem}  already implies the result. 
\end{proof}

\begin{rem} {\em When $n=2$  Corollary~\ref{hnmain} follows from  \cite[Theorem 1]{newstead} and for general $n$ it appeared as \cite[Theorem 1]{harder-narasimhan} of Harder--Narasimhan. In \cite{harder-narasimhan} it was proved by an arithmetic study of $\calN$. Later it was reproved with the gauge theoretical
approach of Atiyah--Bott in \cite[Proposition 9.7]{atiyah-bott}. Our proof above
is by identifying the characteristics for which the  Prym variety is connected, which excludes the nilpotent cone, where $\calN$ is located, and yields Corollary~\ref{hnmain} without focusing on $H^*(\calN;\QQ)$. This approach
thus can be considered as an example of the abelianization philosophy of 
Atiyah--Hitchin \cite[\S 6.3]{atiyah}, by studying the abelian problem of  Prym varieties of spectral curves we deduced results on the  moduli space $\calN$ of vector bundles of rank $n>1$.

}
\end{rem}


\begin{thebibliography}{999}

\bibitem[At]{atiyah} M. Atiyah: {\em Geometry and physics of knots}, CUP 1990

\bibitem[AB]{atiyah-bott}{M.F. Atiyah,   R. Bott}: 
The Yang-Mills equations over Riemann surfaces, 
{\sl Philos.\ Trans.\ Roy.\ Soc.\ London Ser.\ A} 308 (1982) 523--615.

\bibitem[BNR]{BNR} A. Beauville, M.S. Narasimhan, S. Ramanan: Spectral curves and the generalised theta divisor, 
J. reine angew. Math. 398, 169-179 (1989)

\bibitem[BL]{BL} Ch. Birkenhake, H. Lange: Complex Abelian Varieties (Second edition), Grundlehren
der mathematischen Wissenschaften Vol. 302, Springer (2004)

 \bibitem[dCHM]{dCHM} M.A. de Cataldo, T. Hausel, L. Migliorini: Topology of Hitchin systems and Hodge theory of character varieties: the case $A_1$, (preprint arXiv:1004.1420)

\bibitem[CL]{CL} P.-H. Chaudouard, G. Laumon: Le lemme fondamental pond\'er\'e. II. \'Enonc\'es cohomologiques, (preprint arXiv:0912.4512)


\bibitem[D1]{Drezet} J.-M. Dr\'ezet: Faisceaux coh\'erents sur les courbes multiples, Collectanea Mathematica 57, 2 (2006), 121-171

\bibitem[D2]{D} J.-M. Dr\'ezet: Faisceaux sans torsion et faisceaux quasi localement libres sur les courbes multiples primitives, Mathematische Nachrichten, Vol. 282, No.7 (2009), 919-952

\bibitem[FW]{FW} E. Frenkel, E. Witten: Geometric endoscopy and mirror symmetry. 
Commun. Number Theory Phys. 2 (2008), no. 1, 113Ð283. 

\bibitem[GHS]{GHS} O. Garc\'ia-Prada, J. Heinloth, A. Schmitt: On the motives of moduli of chains and Higgs bundles, (preprint arXiv:1104.5558v1)



\bibitem[G1]{EGA2} A. Grothendieck: \'El\'ements de g\'eom\'etrie alg\'ebrique 
(r\'edig\'es avec la collaboration de Jean Dieudonn\'e): II. \'Etude globale \'el\'ementaire de quelques classes de morphismes. Publications Math\'ematiques de l'IH\'ES, 8 (1961), 5-222 

\bibitem[G2]{EGA4} A. Grothendieck: \'El\'ements de g\'eom\'etrie alg\'ebrique 
(r\'edig\'es avec la collaboration de Jean Dieudonn\'e) : IV. \'Etude locale des sch\'emas et des morphismes de sch\'emas, Quatri\`eme partie. Publications Math\'ematiques de l'IH\'ES, 32 (1967), 5-361 

\bibitem[HN]{harder-narasimhan}{ G. Harder}; M.S. Narasimhan:  On the cohomology groups of moduli spaces of vector bundles on curves. {Math. Ann.}  {212} (1974/75)

\bibitem[Har]{hartshorne} { R. Hartshorne}: {\em Algebraic geometry}, 
Graduate Texts in Mathematics 52, Springer-Verlag, New York, 1977.

\bibitem[Hau1]{hausel-compact}
{\sc T.~Hausel}:
\newblock Compactification of moduli of {H}iggs bundles,
\newblock { J. Reine Angew. Math.}, {503} (1998) 169-192.

\bibitem[Hau2]{hausel-survey}{ T. Hausel}: \newblock  Global topology of the Hitchin system, preprint arXiv:1102.1717

\bibitem[HT]{hausel-thaddeus} { T. Hausel, M. Thaddeus}: \newblock 
Mirror symmetry, Langlands duality, and the Hitchin System \newblock { Invent. Math.}. {153} (2003), 197--229. 


\bibitem[Hi]{hitchin}
N.~J. Hitchin:
\newblock Stable bundles and integrable systems.
\newblock { Duke Math. J.}, 54(1):91--114, 1987.

\bibitem[I]{iversen}B. Iversen: \newblock Cohomology of Sheaves. Universitext. Springer (1985).

\bibitem[KW]{kapustin-witten} A. Kapustin, E. Witten: Electric-magnetic duality and the geometric Langlands program. Commun. Number Theory Phys.  1  (2007),  no. 1, 1--236. 

\bibitem[Kl]{K} S. Kleiman: The Picard scheme,  Fundamental algebraic geometry, Math. Surveys Monogr., 123, Amer. Math. Soc., Providence, RI (2005), 235-321

\bibitem[Ki]{kirwan} F.C. Kirwan: {\em Cohomology of quotients in
symplectic and algebraic geometry}, Mathematical Notes 31, Princeton University
Press, 1984

\bibitem[N]{newstead}{ P.E. Newstead}: {Characteristic classes of stable bundles of rank $2$ over an algebraic curve,}  { Trans. Amer. Math. Soc.}  { 169 } (1972), 337--345. 

\bibitem[L]{Liu} Q. Liu: Algebraic geometry and arithmetic curves, Oxford Graduate Texts in Mathematics, No. 6 (2002) 

\bibitem[Ma]{macdonald} { I.G. Macdonald}, Symmetric products of an algebraic
curve, { Topology} {\bf 1} (1962) 319-343



\bibitem[Na]{nakajima} { H.~Nakajima}:  \newblock
{\em Lectures on Hilbert schemes of points on surfaces.}
AMS, Providence, RI, 1999.

\bibitem[N1]{ngo1} B.C. Ng\^o: Fibration de Hitchin et endoscopie, Inventiones mathematicae, Vol. 164 (2006), 399-453

\bibitem[N2]{ngo2} B.C. Ng\^o: Le lemme fondamental pour les alg\`ebres de Lie, 
Publ. Math. Inst. Hautes \'Etudes Sci. No. 111 (2010), 1-169. 

\bibitem[Ni]{nitsure}
N. Nitsure: Moduli space of semistable pairs on a curve,  Proc. London Math. Soc. (3), (1991) {\bf 62}(2): 275--300

\bibitem[Sch]{Schaub} D. Schaub: Courbes spectrales et compactifications de jacobiennes, Math. Zeitschrift, Vol. 227 (1998),
295-312

\bibitem[S1]{Simpson1} C. Simpson: Moduli of representations of the fundamental group of a smooth projective variety I,
Publications Math\'ematiques de l'IH\'ES, 79 (1994), 47-129
 
\bibitem[S2]{Simpson2} C. Simpson: Moduli of representations of the fundamental group of a smooth projective variety II,
Publications Math\'ematiques de l'IH\'ES, 80 (1994), 5-79

\bibitem[S3]{Simpson3} C. Simpson: The Hodge filtration on
nonabelian cohomology, Algebraic geometryÑSanta Cruz 1995, 217Ð281, 
Proc. Sympos. Pure Math., 62, Part 2, Amer. Math. Soc., Providence, RI, 1997. arXiv: alg-geom/9604005


\end{thebibliography}
\end{document}